%% file: root.tex
\documentclass[letterpaper, 10 pt, conference]{ieeeconf}  

\IEEEoverridecommandlockouts                              

\usepackage{preamble}
\usepackage{graphicx}
\usepackage{amssymb}  
\renewcommand{\geq}{\geqslant}
\renewcommand{\leq}{\leqslant}

\title{\LARGE \bf
  Integrating Aggregated Electric Vehicle Flexibilities in Unit Commitment Models using Submodular Optimization
}

\author{Hélène Arvis$^{1,2,3}$, Olivier Beaude$^{1}$, Nicolas Gast$^{2}$, Stéphane Gaubert$^{3}$ and Bruno Gaujal$^{2}$
\thanks{*This work was supported by the Défi Inria-EDF}
\thanks{$^{1}$EDF Lab Paris-Saclay, OSIRIS department,
        {\tt\scriptsize \{Firstname.Name\}@edf.fr}}
\thanks{$^{2}$Univ. Grenoble Alpes, Inria, CNRS, Grenoble INP, LIG,
        {\tt\scriptsize \{Firstname.Name\}@inria.fr}}
\thanks{$^{3}$Inria, and CMAP, \'Ecole polytechnique, IP Paris, CNRS,
  {\tt\scriptsize  \{Firstname.Name\}@inria.fr}}
}

\begin{document}

\maketitle
\thispagestyle{empty}
\pagestyle{empty}

\begin{abstract}
  The \ac{UC} problem consists in controlling a large fleet of
  heterogeneous electricity production units in order to minimize
  the total production cost while satisfying consumer demand. \acp{EV} are used as a source of flexibility and are often aggregated for problem tractability. 
We develop a new approach to integrate \ac{EV} flexibilities in the \ac{UC} problem and exploit the generalized polymatroid structure of aggregated flexibilities of a large population of users to develop an exact optimization algorithm, combining a cutting-plane approach and submodular optimization.
We show in particular that the \ac{UC} can be solved exactly
in a time which scales linearly, up to a logarithmic factor, in the number of \ac{EV} users when each production unit is subject to convex constraints.
We illustrate our approach by solving
  a real instance of a long-term \ac{UC} problem, combining open-source data of the European grid (European Resource Adequacy Assessment project)
  and data originating from a survey of user behavior of the French \ac{EV} fleet.
\end{abstract}

\section{INTRODUCTION}
\subsection{Context}

As the installed capacity for renewable energy sources increases around the world, so does the need for flexibility of the electric power system to absorb their fluctuations. In particular, these flexibilities, including hydropower or \acfp{EV}. must be correctly modeled in long-term optimization problems to evaluate their contribution to the system. \acf{UC} problems are a focus of study in this context and consist in dispatching electricity to balance demand and production while respecting the constraints of large number of heteregeneous productions units,
generally including nuclear power plants (subject to specific dynamic constraints), thermal plants, hydroelectric plants, and other renewable enery sources.

The \ac{UC} problem has been studied since the 1940s and progressively enriched by resolution techniques, as well as by the appearance of new production means, as demonstrated by the literature review of Abdou and Tkiouat \cite{abdou_unit_2018}. We focus on the deterministic version of the \ac{UC}, though it can also be studied in the context of uncertainties \cite{tahanan_large-scale_2015}.  

Typically, the \ac{UC} is a very large \ac{MILP}, modeling dispatch decisions over hundreds of time steps with thousands of power plants in a country. A common approach to solving these problems is through aggregation, i.e., replacing a group of similar units by a class representative in order to reduce problem size and compute solutions in reasonable time. 
The clustering of similar units has been explored by Meus et al. \cite{meus_applicability_2018}, who demonstrate the benefits of aggregation in terms of tractability. However, aggregation generally leads to approximate solutions.

In this paper, we consider the specific case of \ac{EV} charging profiles, in particular the dispatch of power through smart charging by a centralized operator. The correct integration of \acp{EV} in long-term, thus high-level, \ac{UC} problems has an impact on realism of solutions. In France in 2035, there will be an estimated 15.6 million light-duty \acp{EV} \cite{rte_integration_2019}; if all vehicles charge simultaneously at 6kW, this amounts to a national demand of 93.6 GW. In comparison, the French installed production capacity is estimated to reach around 210GW (with 90GW of controllable means) in 2035 \cite{rte_bilan_2023}, thus would be heavily impacted by uncoordinated charging.
Hence, our goal is to integrate an accurate and tractable model of \ac{EV} flexibilities in global \ac{UC} formulations.

Aggregating a large number of flexibilities boils down to computing, or approximating, the Minkowski sum of a large number of polytopes, where each summand represents the flexibility set of an individual user
or group of similar users. 
This Minkowski sum generally has exponentially many inequalities; moreover, calculating the Minkowski sum of a family of polytopes defined by their facets is generally NP-hard \cite{tiwary_hardness_2008}.
Therefore, techniques must be explored to avoid an explicit computation of this Minkowski sum. The traditional, ``direct'', aggregation method consists in summing the characteristics of individual users so as to have a "super-user" that represents the whole population. This generally leads to an aggregation error.
Special cases in which an exact and tractable aggregation is possible have been
identified, by Beaude et al.~\cite{beaude_privacy-preserving_2020} where each local constraint set
has the structure of a (bipartite) network flow polytope, and recently
by Mukhi et al.~\cite{mukhi_exact_2025}, where each local constraint set is shown to have a generalized
polymatroid structure.

\subsection{Contribution}

We study a general \ac{UC} model, including both power production units and a detailed model of \ac{EV} flexibilities. This model captures charging constraint sets over a one week horizon, with multiple charging slots and specific energy consumption constraints (arising for instance from daily commuting).
We leverage the generalized-polymatroid structure of the \ac{EV}-component of the \ac{UC} problem to develop an exact algorithm, combining a cutting-plane approach with submodular optimization.
We show in particular that, when the constraint sets of the power plants are convex, the \ac{UC} problem can be solved to optimality in a time which scales as $N\log N$, where $N$ is the number
of \ac{EV} user profiles (see~\Cref{coro-synthesis}).
We also present a practical cutting-plane algorithm
(\Cref{cutting_planes})
and demonstrate its efficiency on a realistic case study, solving a long-term \ac{UC} problem with European grid data (European Resource Adequacy Assessment project \cite{entso-e_eraa_nodate}) and \ac{EV} data taken from a
survey of French driver behavior \cite{ministere_de_lenvironnement_sdes_enquete_2021}.
In particular, the algorithm scales well relative to the number of iterations and computed cutting planes.

\subsection{Related work}

The specific case of \ac{EV} aggregation has already been studied, though not always in the context of a large \ac{UC} problem. A local use case is the resource allocation problem in which an aggregator must optimize a global cost under given \ac{EV} constraints and a global power bound. Müller et al. \cite{muller_aggregation_2015} offered an alternative to direct aggregation with an inner approximation of the Minkowski sum relying on zonotopes, themselves stable by Minkowski sum. 
Jacquot et al.~\cite{jacquot_privacy-preserving_2019, beaude_privacy-preserving_2020}
studied a similar \ac{EV} resource allocation problem with the additional constraint of preserving the privacy of individuals. To do so, they developed an exact aggregation method exploiting the combinatorial properties of the model (Hoffman cuts for bipartite network flows).

More recently, the aggregated power allocation problem in the context of smart charging (no \ac{V2G}) was found to have the structure of a generalized permutahedron by Mukhi and Abate~\cite{mukhi_exact_2023},
ensuring an exact aggregation of \ac{EV} constraint sets. Panda and Tindemans~\cite{panda_efficient_2024} derived a different approach on a more restrictive use case of homogeneous charging slots based on a concept of UL-flexiblity elaborated by the authors. This notion exploits the permutability of charging patterns to allow for stability by Minkowski sum, thus giving an efficient and scalable approximation. 

Mukhi et al.  \cite{mukhi_exact_2025} observed that \ac{EV} flexibility sets have the structure of generalized polymatroids, a fundamental class of polyhedra introduced by Frank~\cite{frank_generalized_1984} (see also \cite{frank_generalized_1988}). They are special cases of the polyhedra associated to unimodular systems, studied in the broader setting of discrete convexity by Danilov and Koshevoy \cite{danilov_discrete_2004}. Mukhi et al. developed an exact disaggregation method based on generalized polymatroid properties. 
Independently (before the appearance of \cite{mukhi_exact_2025}) Koshevoy pointed out to us~\cite{koshevoy_personnal_2025} that polyhedra associated to the discrete convexity theory of the root system $A_n$, i.e., generalized polymatroids,
 can be used to describe \ac{EV} constraint sets,  leading us to a similar approach.
A more recent work~\cite{mukhi_polymatroidal_2025} extends the work of \cite{mukhi_exact_2025} to model network constraints.

However, the models of~\cite{mukhi_exact_2025, mukhi_polymatroidal_2025} only deal with \ac{EV} charging. In this case, the optimization problem reduces to minimizing a linear cost over a generalized polymatroid, which can be done exactly by the greedy algorithm. 
Here, we consider the full \ac{UC} problem, of which \acp{EV} are only a component. We show that this problem
remains tractable, albeit by more complex methods (practical cutting-plane approaches), using submodular optimization to obtain separation oracles.

Finally, Angeli et al. \cite{angeli_exact_2023} undertook the aggregation of \ac{EV} flexibilities in the \ac{UC} problem in the case of unique charging sessions per vehicle by a constraint-selecting mechanism which can be approximated by a greedy algorithm. 

\subsection{Organisation of the paper}
In \Cref{sec:UC}, we present the \ac{UC} problem and the \ac{EV} constraint set. In \Cref{sec:gpoly}, we recall basic results of generalized polymatroids and their applications to \acp{EV}, along the lines of \cite{koshevoy_personnal_2025} and \cite{mukhi_exact_2025}. In \Cref{sec-theo-results}, we present theoretical results on the complexity of solving the \ac{UC} problem with an exact aggregation of \acp{EV}, as well as present a practical cutting-plane algorithm for this optimization.
The corresponding proofs can be found in 
an extended version of the paper.
Numerical results are presented and discussed in \Cref{sec-benchmark}. Finally, we offer some concluding remarks in \Cref{sec-conclusion}.

\section{THE \ac{UC} PROBLEM AND \ac{EV} FLEXIBILITIES}\label{sec:UC}%
We consider a \acl{UC} problem on a discrete time horizon $\T = \{1, ..., T\}$ with time steps of duration $\tau$. In practice, the horizon is often one week, while the time step is equal to an hour. 

\subsection{Modeling the \acl{UC} problem}

Consider a set $\mathcal{M}$ of production units. Each element $m\in\mathcal{M}$ represents a single unit, which may be a nuclear or thermal (gas, coal, etc.) power plant, as well as a hydroelectric plant.
The production of unit $m$ is represented by a vector $\prodvarv^m=(\prodvar^m_t)_{t\in \T}\in \R^{\T}$, where $\prodvar^m_t$ represents the production at time $t$.

We require the vector $\prodvarv^m$ to belong to a {\em constraint set} $\prodconstraintset^m$, which can differ according to the type of production unit.
This constraint set generally includes upper and lower bounds on production at each instant, i.e., $\underline{k}_t^m \leq \prodvar_t^m \leq \overline{k}_t^m$ for given bounds $(\underline{k}^m_t)_{t\in\T}, (\overline{k}^m_t)_{t\in\T} \in \R^\T$.
For instance, nuclear units are subject to ramping constraints, as well as to minimal on/off duration constraints. The latter lead to \ac{MILP} formulations, rendering the corresponding set $\prodconstraintset^m$ non-convex.

Variable (uncontrollable) renewable energy sources (solar, wind) are integrated via a residual demand $\vect{D}\in \R^{\T}$, obtained by subtracting the renewable production from the power consumption, at every instant $t$.
Finally, we consider an aggregate of \ac{EV} flexibilities with constraint set $\constraintset$, 
which aggregates the behavior of a large heterogeneous population of users.
We will elaborate on this set in the next section.

The \ac{UC} problem is of the form
\begin{subequations}
\begin{alignat}{4}
&&& \min_{\prodvarv, \powerv}& \ \sum_{m\in\mathcal{M}, t\in\T} c_t^m \prodvar_t^m &   \quad            \quad && \label{eq: uc_objective}\\
&&& \text{s.t.}   \quad & \sum_{m} \prodvar_t^m &= D_t + \power_t,  \quad && \forall t\in\mathcal{T} \label{eq: uc_demand}\\
& \quad &&&  \prodvarv^m &\in \prodconstraintset^m, && \forall m\in \mathcal{M} \label{eq: uc_prod_constraints} \\
& \quad &&&  \powerv &\in \constraintset \enspace .&& \label{eq: uc_ev_constraints}%
\end{alignat}
\label{eq: UC_problem}%
\end{subequations}
Here, $c^m_t$ denotes the unitary production cost
at time $t$ for the production plant $m$ so that the total production cost is linear in the dispatching production variables $\prodvarv^m$. The variable $\powerv\in\R^\T$ represents the charging power (positive or negative, in the case of \ac{V2G}) of the fleet of \acp{EV}. Finally, the relation \eqref{eq: uc_demand} expresses the balance between production and demand at every time step.

\subsection{Modeling \ac{EV} flexibilities}\label{sec-ev_constraints}%
We consider \ac{EV} constraints over the whole time horizon. For an individual, this requires the knowledge of charging availability slots, maximal charging power, and energy need at the end of each slot. 
Consider an \ac{EV} $n\in\N$ and its charging profile $\powerv^n$. At each time step $t\in\T$, we have 
\begin{equation} \label{ev_power_bounds}%
\underline{\power}^n_t \leq \power^n_t \leq \overline{\power}^n_t \enspace .
\end{equation}
If the \ac{EV} is plugged in at time $t$, then $\overline{\power}^n_t$ is equal to the maximal charging power available, typically based on charger limits, and $\underline{\power}^n_t$ is equal to the maximal discharging power (negative). We include \ac{V2G} mechanisms, allowing vehicles to inject power back into the grid. In the absence of \ac{V2G} technology, $\underline{\power}^n_t \equiv 0$. If the \ac{EV} is not plugged in at $t$, either because it is on the road or does not have access to a charging point, we have $\overline{\power}^n_t = \underline{\power}^n_t = 0$.

At each instant, the \ac{EV} also has a \ac{SoC} $\soc^n_t$, bounded by the battery capacity $\overline{\soc}^n_t$ and the minimal energy required at time $t$, $\underline{\soc}^n_t$, giving that
 \begin{equation} \label{ev_soc_bounds}%
\underline{\soc}^n_t \leq \soc^n_t \leq \overline{\soc}^n_t\enspace .
\end{equation}
 The minimal energy required at time $t$ is typically the energy required by the user to complete his or her next trip. 
 
Finally, define $\gamma^n_t \geq 0$ the energy consumed by driving during time step $t$.
We can link $\vect{\soc}^n$ and $\powerv^n$ by

\begin{equation}\label{ev_state_eq}%
\soc^n_{t} = \soc^n_{t-1} + \powerv^n_t - \gamma^n_t, \qquad \forall t\in\T 
\end{equation}
where $\soc_0^n$ is the initial state-of-charge of the given user. 
We note that $\vect{p}^n$ and $\vect{\soc}^n$ are not homogeneous up to a conversion coefficient on $\vect{p}^n$; however, this coefficient has no impact on the following mathematical results and is ignored.

We can see that \eqref{ev_soc_bounds} and \eqref{ev_state_eq} can be combined into a constraint on the cumulative injected power:
\begin{equation}\label{ev_cum_bounds}%
\underline{s}^n_t \leq \sum_{t'= 1}^t \power_{t'}^n  \leq \overline{s}^n_t \qquad \forall t\in\T \enspace ,
\end{equation}
with $\underline{s}^n_t = \underline{\soc}^n_t - \soc_0^n + \sum_{t'=1}^t \gamma^n_{t'}$ and $\overline{s}^n_t = \overline{\soc}^n_t - \soc_0^n + \sum_{t'=1}^t \gamma^n_{t'}$.

In this case, the individual \ac{EV} constraint set is defined by 
\begin{equation}
\constraintset^n = \left\{ \powerv^n \in \R^\T \middle| \powerv^n \text{ respects \eqref{ev_power_bounds} and \eqref{ev_cum_bounds}}
\right\} .
\end{equation}

\subsection{Aggregating \ac{EV} constraints}

The power profile set $\mathcal{P}$ of constraint \eqref{eq: uc_ev_constraints} arising from the flexible \acp{EV} is in fact the Minkowski sum of the individual power profile sets of the \acp{EV}: $\mathcal{P}=\minko$. In general, this set has a very rich face description whose complexity can grow quickly with $N$. This makes the \ac{UC} \eqref{eq: UC_problem} hard to solve. 

To cope with the complexity of $\mathcal{P}$, a classical method is to aggregate the constraint sets by using what we call a \emph{naive} aggregation. This naive aggregation set $\aggreg$ is defined by summing the individual constraints:
\begin{equation}\label{aggreg_set}%
\aggreg = \left\{\power\in\R^\T \middle | \begin{array}{ll}
 \sum\limits_{n\in\N}\underline{\power}^n_t \leq \power_t \leq \sum\limits_{n\in\N}\overline{\power}^n_t,  \qquad &\forall t\in\T \\
\sum\limits_{n\in\N}\underline{s}^n_t \leq \sum\limits_{t'= 1}^t \power_{t'}  \leq \sum\limits_{n\in\N}\overline{s}^n_t,  &\forall t\in\T
\end{array} \right\}
\end{equation}

Note that the aggregation method defines at most $4T$ constraints, thus $\aggreg$ has at most $4T$ facets. 

It is immediate that $\aggreg$ is an outer approximation of $\minko$.
The inverse is not true, as taking the Minkowski sum creates ``mixed'' facets,
not obtained by a naive summation.
However, we shall see that one can leverage the special character of the defining constraints~\eqref{ev_power_bounds},\eqref{ev_cum_bounds} to obtain an exact solution of the \ac{UC} problem.

\section{GENERALIZED POLYMATROIDS AND SUBMODULAR OPTIMIZATION}\label{sec:gpoly}%
In this section, we recall basic results on generalized polymatroids, based on the work of Frank \cite{frank_generalized_1984}\cite{frank_connections_2011} and Danilov and Koshevoy \cite{danilov_discrete_2004}. We also recall their application to \acp{EV} \cite{mukhi_exact_2025}. 
The considered ground set is $\T = \{1, \cdots, T\}$ where $T$ is the planning horizon.

In what follows, when considering a subset $A\subseteq \T$ and a vector $\vect{x} \in\R^\T$, we use the notation $x(A) = \sum_{a\in A} x_a$.

\subsection{Basics of generalized polymatroids}
First, we define special types of set-functions.
\begin{definition}\label{def:submodular}%
  A function $\submod : 2^\T \rightarrow \R\cup \{ +\infty\}$ 
  is
  \textbf{submodular} if
\begin{equation}\label{eq:submodular}%
\!\!\!\! \submod(A) + \submod(B) \geq \submod(A\cap B) + \submod(A\cup B), \quad \forall A, B \subseteq \T \, .
\end{equation}
 A function $\supermod :2^\T \rightarrow \R\cup \{ -\infty\}$ is \textbf{supermodular} if $-\supermod$ is submodular.

\end{definition}
A reformulation of inequality \ref{eq:submodular} gives it an interpretation of a decreasing marginal utility property.

\begin{definition}\label{def:paramodular}%
    Consider a submodular function $\submod$ and a supermodular function $\supermod$ such that $\supermod(\emptyset) = \submod(\emptyset) = 0$. The pair $(\supermod, \submod)$ is said \textbf{paramodular} if the following cross-inequality holds:
    \begin{equation}\label{eq:paramodular}%
     \!\!\!\!\!   \submod(A) - \supermod(B) \geq \submod(A\setminus B) - \supermod(B \setminus A), \quad \forall A, B \subseteq \T \, .
    \end{equation}

\end{definition}

We can note that paramodularity ensures that the space between the two functions is not empty, i.e., 
\begin{align*}
\{\vect{x}\in\R^\T \ &| x(A) \leq \submod(A),\ \forall A \subseteq \T \}\ \cap \\
& \{x\in\R^\T \ |\ \supermod(A) \leq x(A),\ \forall A \subseteq \T \} \neq \emptyset \, .
\end{align*}

Polymatroids,  introduced by Edmonds \cite{guy_submodular_1970},  are polyhedra defined by a submodular function as $\{ \vect{x} \ | x(A) \leq \submod (A), \  \forall A \subseteq \T\}$. 
This is inspired by the theory of matroids in which $f$ is the rank function \cite{whitney_abstract_1935}. 
Frank\cite{frank_generalized_1984} extended this notion to polyhedra defined by a paramodular pair, leading to the notion of {\em generalized} polymatroids. 

\begin{definition}\label{def:gpolymatroid}%
    Consider a paramodular pair of functions $(\supermod, \submod)$. The polyhedron 
    \begin{equation}\label{eq:gpolymatroid}%
     \!\!\!\!   \gpoly(\supermod,\submod) \coloneq \{x\in\R^\T \, |\, \supermod(A) \leq x(A) \leq \submod(A),\ \forall A \subseteq \T \} \!\!\!
    \end{equation}
    is a \textbf{generalized polymatroid}, or a \textbf{g-polymatroid}.
  The pair  $(\supermod,\submod)$ is called the \textbf{border pair}, with $\supermod$ the \textbf{lower border function} and $\submod$ the \textbf{upper border function}.
\end{definition}

Generalized polymatroids have also appeared in work by Danilov and Koshevoy in the broader setting of discrete convexity, dealing with polyhedra whose edges are parallel to a unimodular collection of vectors.
In particular, generalized polymatroids correspond to the special unimodular family
consisting of the vectors $e_i -e_j$ and $e_i$ for $i,j\in \T$ originating from the root system
$A_n$; we refer the reader to \cite{danilov_discrete_2004} for more information.

According to \cite[Prop. 2.5]{frank_generalized_1988}, the border pair of a given g-polymatroid $\gpoly \neq \emptyset$ is unique. In particular, we have $\submod(A) = \max_{x\in \gpoly} x(A)$ and $\supermod(A) = \min_{x\in \gpoly} x(A)$ for every subset $A \subseteq \T$. 

One of the main advantages of these polytopes is the following property:

\begin{theorem}{(See \cite[Thm. 14.2.15]{frank_connections_2011}, \cite[Sec. 7]{danilov_discrete_2004})}\label{thm:minko_sum_gpoly}
Consider paramodular pairs $(\supermod^i, \submod^i), \ i\in [k]$. Then 
\begin{equation}\label{eq:minko_sum_gpoly}
    \sum_{i} \gpoly(\supermod^i, \submod^i) = \gpoly(\sum_i \supermod^i, \sum_i \submod^i) \enspace .
\end{equation}

    In particular, g-polymatroids are stable by Minkowski sum. 
\end{theorem}

The following key observation ensures that EV-flexibilities
are modeled by generalized polymatroids.

\begin{theorem}\label{thm:EV_g_poly}
For each \ac{EV} $n\in\N$, the set $\constraintset^n$ is a g-polymatroid.
\end{theorem}
\begin{proof}
Koshevoy \cite{koshevoy_personnal_2025} pointed out to us that 
this follows from the laminar character of the family of constraints, using the properties given in Examples 5 and 12 of \cite{danilov_discrete_2004}. Mukhi et al.~\cite{mukhi_exact_2025} provide a different proof based on an explicit computation of the border functions.
\end{proof}

Without loss of generality, we assume that the initial collection of constraints \eqref{ev_power_bounds} and \eqref{ev_cum_bounds} is tight. By this, we mean that, considering the border pair
$(\supermod^n, \submod^n)$ such that $\constraintset^n = \gpoly(\submod^n, \supermod^n)$, we have $\submod^n(\{t\}) = \overline{p}^n_t$ and $\supermod^n(\{t\}) = \underline{p}^n_t$, as well as $\submod^n(\{1, \cdots, t\}) = \overline{s}^n_t$ and $\supermod^n(\{1, \cdots, t\}) = \underline{s}^n_t$ for each $t\in\T$.   We denote by $(\supermod,\submod) = (\sum_n \supermod^n, \sum_n\submod^n)$ the border function of the Minkowski sum $\sum_n \mathcal{P}^n$,
which is a g-polymatroid, by~\Cref{thm:minko_sum_gpoly} and~\Cref{thm:EV_g_poly}.

\begin{remark}

When optimizing a linear cost over g-polymatroids, the greedy algorithm gives an exact solution
(see \cite{frank_connections_2011}). This setting is described in \cite[Sec. IV]{mukhi_exact_2025} but does not suffice when the g-polymatroid is integrated into a larger model in the \ac{UC} problem. 
\end{remark} 

To tackle the larger \ac{UC} problem, we rely on
base polyhedra, which are special cases of g-polymatroids.

\begin{definition}\label{def:base_polyhedron}
Consider $\tilde{\T} =\mathcal{T} \cup \{T+1\}$ and a submodular function $\submod$ such that $\submod(\tilde{\T})$ is finite. The \textbf{base polyhedron} with border function $\submod$ is the set 
\begin{equation}\label{eq:base_polyhedron}
    \basepoly(\submod) \coloneq \left\{x\in\R^{\tilde{\T}} \ \middle | 
\begin{array}{l}
x(\tilde{\T}) = \submod(\tilde{\T}) \\
x(A)\leq \submod(A),\  \forall A\subseteq \tilde{\T} \end{array}
 \right\} \enspace .
\end{equation}
A \textbf{0-base polyhedron} is a base polyhedron such that $\submod(\tilde{\T})~=~0$.
\end{definition}

We can easily show that a base polyhedron is a g-polymatroid by taking $\supermod(A) \coloneq \submod(\tilde{\T}) -\submod(\tilde{\T} \setminus  A)$ \cite{frank_connections_2011}. 
In addition, it has been shown that every g-polymatroid is the projection of a 0-base polyhedron \cite[Thm. 14.2.5]{frank_connections_2011}. 
Therefore, 
optimizing over a g-polymatroid reduces to optimizing over the corresponding base polyhedron.
These polyhedra can easily be described as, given a paramodular pair $(\supermod, \submod)$, the submodular function defining the base polyhedron is given by  
\begin{equation}\label{submod_to_base}
\submod^*(A) = \left\{\begin{array}{ll}
        \submod(A) & \text{ if } A\subseteq \T \\
        -\supermod(\T \setminus A) & \text{ if } T+1 \in A 
    \end{array} \right. \enspace .
\end{equation}

\section{SOLVING THE \ac{UC} PROBLEM WITH \acp{EV}}\label{sec-theo-results}

We next show that the \ac{UC} problem integrating \ac{EV} charging flexibilities can be solved efficiently using an approach combining cutting planes and submodular optimization.
In \Cref{sec-theory}, we provide a theoretical complexity result.
Furthermore, we give in \Cref{sec-practical}
a practical cutting-plane approach, used to solve our real instances.

\subsection{Polynomial-time solvability of the convex \ac{UC} problem}\label{sec-theory}

In this subsection, we suppose that the constraint sets $\prodconstraintset^m$
are polyhedra given by collections of linear inequalities,
so  that the production $z^m$ of unit $m$ must belong
to the set
$\prodconstraintset^m=\{z^m\in \R^\T\mid F^m z^m\leq b^m\}\subset \mathbb{R}^{\mathcal{T}}$,
where $F^m$ is a $r_m\times T$ matrix and $b^m$ is a column vector of dimension $r_m$,
where $r_m$ is the number of inequalities. 
This allows both for bound constraints and dynamic linear constraints on the production
but excludes minimal on/off duration constraints for nuclear plants
(which lead to nonconvex -- in fact \ac{MILP} -- formulations). We refer to
this problem as the {\em convex} \ac{UC} problem. 
The convexity assumption is useful to get a theoretical complexity result; it will
be relaxed in the practical method of \Cref{sec-practical}.

Our approach relies on the analysis of linear programs defined implicitly by separation oracles, which goes back to the work of Gr\"{o}tschel, Lovász and Schrijver~\cite{grotschel_geometric_1988}. To obtain a tighter explicit bound, we use a recent
improvement of their approach~\cite{dadush_finding_2022} relying on
the cutting-plane method of~\cite{jiang_improved_2020}.

A key notion is that of a {\em polyhedral separation oracle}~\cite{dadush_finding_2022}. Given a polyhedron $P$ in $\R^d$ defined by a finite collection of inequalities $(G_ix\leq h_i)_{i\in \mathcal{I}}$,
with $G_i \in \R^{1\times d}$, $h_i \in \R$,
a {\em polyhedral separation oracle} is an algorithm
which, given a vector $\vect{x}\in\R^D$, asserts that $\vect{x}\in P$ or returns
a violated inequality, i.e., an index $i\in \mathcal{I}$ such that $G_i \vect{x}>h_i$.
Note that this notion is more demanding than the classical notion
of strong separation oracle, used in~\cite{grotschel_geometric_1988}.
The number of inequalities $|\mathcal{I}|$ is allowed to be exponential in $d$.

We work in the Turing model of computation, so we assume that all of the problem inputs are rational vectors. Without loss of generality, after multiplying by common denominators, we will even assume that the entries are integers. Then, we denote by $B$ the maximal absolute value of all the integer coefficients appearing in the problem input, i.e., of the entries of the matrix $F^m$, of the vector $\vect{b}^m$, and of the bounds
$\underline{s}^n_t$, $\overline{s}^n_t$, $\underline{p}^n_t$ and $\overline{p}^n_t$. The {\em dimension}
(total number of scalar variables) is $d = (|\mathcal{M}| + 1)T$.
The following theorem, along with its corollary below, results from that of Dadush et al.~\cite[p.~2703, para.~2, line~4]{dadush_finding_2022} and
is proved in an extended version of this paper.  

\begin{theorem}\label{th-complexity}
  The convex \ac{UC} problem integrating a collection of $N$
  flexible \acp{EV} can be solved in a 
  number of polyhedral separation oracle queries bounded by $O(d^3\log (dBN))) $, with $d$ the dimension and $B$ the maximal absolute value of integer coefficients of the problem input.
\end{theorem}
\begin{remark}
This bound is independent of the bit-size of of the cost vector $\vect{c}^m$. 
\end{remark}
Such an oracle can be implemented
efficiently. Given a vector $x=( (z^m)_{m\in \mathcal{M}}, p)$, 
checking the validity of the inequalities $F^m z^m \leq b^m$ or of the global demand
constraint is immediate (there are $\sum_m r_m + T$ scalar constraints that can be checked
one by one). 
The \ac{EV} polyhedral separation oracle deciding whether $p$ belongs to the Minkowski sum $\sum_n \mathcal{P}^n$,  which is decribed by $2^{T+1}$ g-polymatroid type inequalities. is implemented by a reduction to \ac{SFM}. In fact, since $\minko$ is a g-polymatroid, define $(\supermod, \submod)$ the corresponding border pair. We can then define $\submod^*$ as in \eqref{submod_to_base} and $\basepoly(\submod^*)$ the corresponding 0-base polyhedron. For a vector $\vect{x}\in \R^\T$, define the extended vector $\vect{x}^* = (x_1, x_2, \dots, x_T, -\sum_{t\in\T} x_t) \in \R^{\T+1}$. 
In this case, 
\begin{equation}\label{separation_equiv}
\vect{x}\in\minko \iff \vect{x}^* \in\basepoly(\submod^*) \enspace ,
\end{equation} 
meaning that the membership problem of $\minko$ is equivalent to that of $\basepoly(\submod^*)$, and the separation problem for $\minko$ can be deduced from that for $\basepoly(\submod^*)$.

Given a vector $\powerv\in\R^\T$, the separation problem is given by minimizing the submodular function $\submod^* - \powerv^*$. In fact, if for a minimal subset $A$, $(\submod^* - \powerv^*)(A) \geq 0$, then $\powerv^*\in\basepoly(\submod^*)$ and $\powerv \in \minko$, thus the solution is an exact aggregation. If $(\submod^* - \powerv^*)(A) < 0$, then the constraint $\powerv^*(A) \geq \submod^*(A)$ is violated and yields a cutting plane. The corresponding cut for $\minko$ can be deduced from \eqref{submod_to_base}.

Applying \Cref{th-complexity} and bounds on the resolution of \ac{SFM} \cite{orlin_faster_2009} yields the following result.
\begin{corollary}\label[corollary]{coro-synthesis}
  The convex \ac{UC} problem integrating $N$ flexible \acp{EV} can be solved
  in $N\log N \operatorname{poly}(T,d, \log B, r)$ arithmetic operations,
  where $r=\sum_m r_m$ is the total number of constraints satisfied by production units, and $\operatorname{poly}(\cdot)$ 
  denotes a function polynomially bounded in its arguments.
\end{corollary}

\subsection{The practical cutting-plane algorithm}\label{sec-practical}
Having established the theoretical complexity in \Cref{coro-synthesis}, we present here a practical algorithm to solving \ac{UC} problem with an exact aggregation of \ac{EV} constraint sets.
The algorithm, described in \Cref{cutting_planes}, iteratively solves the \ac{UC} problem \eqref{eq: UC_problem} while enriching the constraint set at each iteration. We build 
a decreasing sequence of polytopes $$\constraintset^{(0)} \supset \constraintset^{(1)}\supset \cdots$$ such that $\constraintset^{(0)} = \aggreg$ the naive aggregation polytope \eqref{aggreg_set},
$\constraintset^{(\infty)} \coloneqq \cap_{k\in\mathbb{N}}\constraintset^{(k)}= \minko$ the exact polytope, and each intermediate polytope has a richer constraint dictionary than the previous one. 

Additionally, consider, for $k\in\mathbb{N}$, the problem 
\begin{tagsubequations}{$UC_k$}
\begin{alignat}{4}
& \min_{\prodvarv, \powerv}& \ \sum_{m\in\mathcal{M}, t\in\T} c_t^m \prodvar_t^m && \label{eq: uck_objective}\\
& \text{s.t.}   & \eqref{eq: uc_demand}, \ \eqref{eq: uc_prod_constraints}&&\\
& \quad &  \powerv \in \constraintset^{(k)} && \label{eq: uc_evk_constraints}
\end{alignat}
\label{uc_iteration_k}%
\end{tagsubequations}
i.e., the \ac{UC} problem with the enriched polytope of iteration $k$. After solving problem \eqref{uc_iteration_k}, we can apply the separation oracle described in \Cref{sec-theory}
to the \ac{EV} solution $\powerv_k$: if the solution is an exact aggregation, the algorithm stops. Otherwise, the separation oracle returns a violated constraint which can be added to build
the enriched constraint set $\constraintset^{(k+1)}$. In practice, running the \ac{SFM} algorithm allows for far more constraints to emerge as it must evaluate $\submod^*$ repeatedly, 
and we use the whole collection of cuts obtained in this way, all at once, to enrich the constraint set, passing from $\constraintset^{(k)}$ to  $\constraintset^{(k+1)}$.

\begin{algorithm}
\caption{A cutting-plane resolution of \ac{UC}}\label{cutting_planes}
$\constraintset^{(0)} \gets \aggreg$\;
$k \gets 0$\;
\For{$k \geq 0$}{
Solve \eqref{uc_iteration_k} and retrieve optimal $\powerv_k$ \;
Get $A \in \argmin (\submod^* - \powerv_k^*)$ by SFM \;
Get $\mathcal{C}_k$ polyhedron defined by facets calculated during SFM \;
\eIf{ $(\submod^* - \powerv_k^*)(A) \geq 0$}{
	$\powerv_k$ is disaggregeable: STOP \;
}{$\constraintset^{(k+1)} = \constraintset^{(k)} \cap \mathcal{C}_k$
}} 
\end{algorithm}

\begin{prop}\label{prop_termination}
  \Cref{cutting_planes} terminates and returns an optimal solution.
\end{prop}

This property results from the g-polymatroid structure, which ensures that $\minko$ is defined by at most $2^{T+1}$ constraints.

\section{NUMERICAL RESULTS OF THE CUTTING-PLANE ALGORITHM}\label{sec-benchmark}

We next present a realistic instance of the European grid and the French \ac{EV} fleet, as well as some numerical results of \Cref{cutting_planes} applied to this instance.

\subsection{The European Grid Instance}\label{sec-eraa-data}

Following the EU Electricity Regulation in 2019, the \ac{ENTSO-E} began work on the \ac{ERAA} as part of a toolbox to improve comprehension of the European power system and ensure that investments and regulations today are coherent with future needs \cite{entso-e_eraa_nodate}.   

In particular, the \ac{ERAA} report provides a relatively complete and open-access dataset on the European instance, describing various aspects of the grid by electricity market node.  In this paper, we exploit the data of the 2023.2 edition.

Using this data, we were able to implement a seven-country \ac{UC} problem centered around the French market node. In this problem, we consider a total of 64 aggregated production units (nuclear, thermal, and hydroelectric plants), renewable energy sources being integrated to the residual demand.

\subsection{Modeling the French fleet of \acp{EV}}

In order to simulate the aggregation of \acp{EV} in the \ac{UC} problem, we used the \ac{EV} data of the \textit{Enquête de Mobilité des Personnes} \cite{ministere_de_lenvironnement_sdes_enquete_2021}, a survey on the behavior of French drivers. From this survey, we can extract different user profiles and generate realistic individual weekly driving schedules with potential charging slots and driving distance, thus energy need. This data directly supplies us with different profiles of \ac{EV} constraint sets described in \Cref{sec-ev_constraints}. 

In order to correctly size the proportion of \acp{EV} present in the \ac{UC} problem, we rely on predictions by RTE \cite{rte_integration_2019}. In this report, several growth scenarios are presented, with the \textit{Crescendo} being the most likely. Considering this scenario, there will be 8.5 million \acp{EV} in France in 2030, 60\% of which should subject themselves to smart charging and supposed here made available in the \ac{UC} problem. As such, we consider a fleet of 5.1 million "controllable" \acp{EV}, which each belong to one of $N$ different profiles chosen at random. 

Finally, we choose $\soc_0$ the initial \ac{SoC} and $\soc_T$ the target \ac{SoC} (minimal level for the end of the optimization horizon demanded by each user) such that the total \ac{EV} power consumption should fall around $3-4\%$ of the total demand of the period.

\subsection{Implementation}

The \ac{UC} problem is implemented in Python 3.12 as a \ac{LP} and solved by \textit{CPLEX} 22.1.2.
The separation oracle is implemented in Julia 1.9.4 and is based on an existing \ac{SFM} package
\cite{lock_submodularminimizationjl_2025}, which implements the Fujishige-Wolfe minimum norm point algorithm \cite{fujishige_submodular_2009} as described by Chakrabarty et al. \cite{chakrabarty_provable_2014}. Though the complexity of this algorithm is theoretically worse than that of \cite{iwata_simple_2009}, it is shown to outperform the latter algorithm in practice.

Finally, the evaluation oracle of $\submod^* - \powerv^*_k$ is implemented as described in the proof of \Cref{coro-synthesis}.
In practice, we implemented a simple parallelization of the calculation of the $N$ \acp{LP} using multithreading.

\subsection{Benchmark}

In order to evaluate the performance of \Cref{cutting_planes}, simulations were run on the projected \ac{ERAA} data from the week, starting on Monday, of January 13, 2030 with the climatic conditions of 2003 (see \Cref{sec-eraa-data}),
over horizons $T\in \{24, 48, 96, 168\}$ and considering different numbers of \ac{EV} user profiles $N\in\{2, 10, 50, 100\}$, where the population of 5.1 million \acp{EV} is evenly distributed among profiles. These simulations were run on 24 cores of an Intel® Xeon® Platinum 8260 processor, with 10 runs for $T=24, 48, 96$ and 5 for $T=168$ \footnote{The case $T=168,\ N=100 $ is not included in the figures because the solver crashed during execution. We suspect that this is due to a known memory leak of the Julia 1.9 using multithreading on Linux.}, each run differing in the selected \ac{EV} profiles.

In this paper, we report numerical results in the convex \ac{UC} setting, corresponding to the available public data, though tests on the "full" \ac{UC} problem (\ac{MILP} model integrating dynamic on/off constraints) yielded similar performance.

These simulations show \Cref{cutting_planes} to terminate in few iterations, as illustrated by \Cref{fig_benchmark_av_iterations}. In fact, all instances required fewer than $8$ iterations to terminate, and most fewer than $5$.
Furthermore, the number of cuts calculated by the separation oracle remained far smaller than the theoretical $2^{T+1}$ bound, as shown by \Cref{fig_benchmark_max_added_cuts}. 
 
The main point of improvement for this algorithm is the runtime, as illustrated by \Cref{fig_benchmark_av_res_time}. The bottleneck is the separation oracle; for $T=168$ and $N=50$, $85\%$ of the computation time is taken by the separation oracle, for instance. 
The multithreading implementation of the evaluation oracle is already faster than a sequential calculation ($8$ times faster on average for $T=96$, $N=50$ for example); exploring other parallelization implementations could further improve the performance. 
This will be the object of future studies.

\begin{figure*}[h]
  \begin{minipage}[t]{.3\linewidth}
\centering
\resizebox{\linewidth}{!}{\input{figures/benchmark_av_iterations.tex}}
\caption{Number of iterations performed by \Cref{cutting_planes}. The error bar represents extreme values obtained, as well as the average value. Despite an increase with problem size, the iteration number remains low.}
\label{fig_benchmark_av_iterations}%
  \end{minipage}\hfil
  \begin{minipage}[t]{.3\linewidth}
\centering
\resizebox{\linewidth}{!}{\input{figures/benchmark_av_added_cuts.tex}}
\caption{Number of cuts added to \ac{UC}. There is no point for $N=2$ at $T=24$ as every simulation converged in 1 iteration. All values are much smaller than the theoretical $2^{T+1}$ bound.}
\label[figure]{fig_benchmark_max_added_cuts}%
  \end{minipage}\hfil
  \begin{minipage}[t]{.3\linewidth}
\centering
\resizebox{\linewidth}{!}{\input{figures/benchmark_av_res_time.tex}}
\caption{Resolution time as a function of $T$ and $N$. As expected, it increases with $T$ and $N$.}
\label{fig_benchmark_av_res_time}%
  \end{minipage}%
\end{figure*}

\section{CONCLUDING REMARKS}\label{sec-conclusion}
We studied the integration of flexible \acp{EV}
in a unit commitment problem.
We showed that this problem can be solved exactly and
in a scalable way for instances 
with a large number $N$ of categories of flexible \ac{EV} users,
with a theoretical complexity scaling in $N\log N$.
We also developed a practical cutting-plane algorithm,
which, on a real large-scale instance, converged in few iterations.
This work can be extended in several ways. First, improvements of the practical algorithm can be considered, such as exploring other strategies to add cutting planes. Next, it would be interesting to
obtain tight bounds on the true number of facets of the polytopes representing
\ac{EV} flexibilities
(experiments in small dimension
indicate that for typical instances, this number is far below the theoretical
$O(2^{T+1})$ bound); parcimonious approximations of these polytopes may also be
investigated.
Finally, it may be interesting to cover other types of flexibilities, such as hydropower.

\addtolength{\textheight}{-1cm}   



%

\section*{ACKNOWLEDGMENT}
We thank Gleb Koshevoy for pointing out the generalized polymatroid structure of \ac{EV} flexibilities and thank Juliette Lesturgie for her help in understanding the power system and \ac{UC} constraints. We also thank both of them for many discussions.


\bibliographystyle{ieeetr}
\bibliography{references}

\end{document}

%% file: figures/benchmark_av_iterations.tex
%
\begin{tikzpicture}

\definecolor{crimson2143940}{RGB}{214,39,40}
\definecolor{darkgray176}{RGB}{176,176,176}
\definecolor{darkorange25512714}{RGB}{255,127,14}
\definecolor{forestgreen4416044}{RGB}{44,160,44}
\definecolor{lightgray204}{RGB}{204,204,204}
\definecolor{steelblue31119180}{RGB}{31,119,180}

\begin{axis}[
legend cell align={left},
legend style={error bar legend, fill opacity=0.8, draw opacity=1, text opacity=1, draw=lightgray204},
tick align=outside,
tick pos=left,
x grid style={darkgray176},
xlabel={Optimization horizon $T$ (h)},
xmin=-0.31, xmax=3.21,
xtick style={color=black},
xtick={0,1,2,3},
xticklabels={24,48,96,168},
y grid style={darkgray176},
ymin=0.65, ymax=8.35,
ytick style={color=black},
ytick={1,...,8},
ylabel={Number of iterations}
]
\path [draw=steelblue31119180, semithick]
(axis cs:-0.15,1)
--(axis cs:-0.15,1);

\path [draw=steelblue31119180, semithick]
(axis cs:.85,1)
--(axis cs:.85,2);

\path [draw=steelblue31119180, semithick]
(axis cs:1.85,1)
--(axis cs:1.85,2);

\path [draw=steelblue31119180, semithick]
(axis cs:2.85,1)
--(axis cs:2.85,3);

\path [draw=darkorange25512714, semithick]
(axis cs:-0.05,1)
--(axis cs:-0.05,4);

\path [draw=darkorange25512714, semithick]
(axis cs:0.95,1)
--(axis cs:.95,3);

\path [draw=darkorange25512714, semithick]
(axis cs:1.95,1)
--(axis cs:1.95,3);

\path [draw=darkorange25512714, semithick]
(axis cs:2.95,2)
--(axis cs:2.95,3);

\path [draw=forestgreen4416044, semithick]
(axis cs:0.05,1)
--(axis cs:0.05,4);

\path [draw=forestgreen4416044, semithick]
(axis cs:1.05,1)
--(axis cs:1.05,4);

\path [draw=forestgreen4416044, semithick]
(axis cs:2.05,2)
--(axis cs:2.05,5);

\path [draw=forestgreen4416044, semithick]
(axis cs:3.05,2)
--(axis cs:3.05,5);

\path [draw=crimson2143940, semithick]
(axis cs:0.15,1)
--(axis cs:0.15,7);

\path [draw=crimson2143940, semithick]
(axis cs:1.15,2)
--(axis cs:1.15,8);

\path [draw=crimson2143940, semithick]
(axis cs:2.15,2)
--(axis cs:2.15,5);

\addplot [semithick, steelblue31119180, mark=*, mark size=2, mark options={solid}, only marks]
table {%
-0.15 1
0.85 1.3
1.85 1.2
2.85 2
};
\addlegendentry{2 EVs}
\addplot [semithick, darkorange25512714, mark=*, mark size=2, mark options={solid}, only marks]
table {%
-0.05 1.7
0.95 1.3
1.95 2.1
2.95 2.4
};
\addlegendentry{10 EVs}
\addplot [semithick, forestgreen4416044, mark=*, mark size=2, mark options={solid}, only marks]
table {%
0.05 2.6
1.05 1.9
2.05 3.4
3.05 3.4
};
\addlegendentry{50 EVs}
\addplot [semithick, crimson2143940, mark=*, mark size=2, mark options={solid}, only marks]
table {%
0.15 3.8
1.15 3.2
2.15 3.3
};
\addlegendentry{100 EVs}
\end{axis}

\end{tikzpicture}

%% file: figures/benchmark_av_added_cuts.tex

\begin{tikzpicture}

\definecolor{crimson2143940}{RGB}{214,39,40}
\definecolor{darkgray176}{RGB}{176,176,176}
\definecolor{darkorange25512714}{RGB}{255,127,14}
\definecolor{forestgreen4416044}{RGB}{44,160,44}
\definecolor{lightgray204}{RGB}{204,204,204}
\definecolor{mediumpurple148103189}{RGB}{148,103,189}
\definecolor{steelblue31119180}{RGB}{31,119,180}

\begin{axis}[
legend cell align={left},
legend style={ fill opacity=0.8, draw opacity=1, text opacity=1, draw=lightgray204},
legend entries={$2^{T+1}$, 2 EVs, 10 EVs, 50 EVs, 100 EVs},
log basis y={10},
tick align=outside,
tick pos=left,
x grid style={darkgray176},
xlabel={Optimization horizon $T$ (h)},
xmin=10.4, xmax=177.6,
xtick style={color=black},
xtick={24,48,...,168},
y grid style={darkgray176},
ymin=0.5, ymax=1e+12,
ymode=log,
ytick style={color=black},
ylabel={Number of added cuts}
]
\addlegendimage{no markers, mediumpurple148103189}
\addlegendimage{error bar legend, only marks, mark size = 1, steelblue31119180}
\addlegendimage{error bar legend, only marks, mark size = 1, darkorange25512714}
\addlegendimage{error bar legend, only marks, mark size = 1, forestgreen4416044}
\addlegendimage{error bar legend, only marks, mark size = 1, crimson2143940}
\path [draw=steelblue31119180, semithick]
(axis cs:18,0.01)
--(axis cs:18,0.01);

\path [draw=steelblue31119180, semithick]
(axis cs:42,0.01)
--(axis cs:42,4184);

\path [draw=steelblue31119180, semithick]
(axis cs:90,0.01)
--(axis cs:90,10950);

\path [draw=steelblue31119180, semithick]
(axis cs:162,0.01)
--(axis cs:162,772704);

\path [draw=darkorange25512714, semithick]
(axis cs:22,0.01)
--(axis cs:22,5360);

\path [draw=darkorange25512714, semithick]
(axis cs:46,0.01)
--(axis cs:46,13360);

\path [draw=darkorange25512714, semithick]
(axis cs:94,0.01)
--(axis cs:94,155215);

\path [draw=darkorange25512714, semithick]
(axis cs:166,70644)
--(axis cs:166,408855);

\path [draw=forestgreen4416044, semithick]
(axis cs:26,0.01)
--(axis cs:26,20350);

\path [draw=forestgreen4416044, semithick]
(axis cs:50,0.01)
--(axis cs:50,65184);

\path [draw=forestgreen4416044, semithick]
(axis cs:98,44642)
--(axis cs:98,392021);

\path [draw=forestgreen4416044, semithick]
(axis cs:170,0.01)
--(axis cs:170,681137);

\path [draw=crimson2143940, semithick]
(axis cs:30,0.01)
--(axis cs:30,76943);

\path [draw=crimson2143940, semithick]
(axis cs:54,0.01)
--(axis cs:54,254282);

\path [draw=crimson2143940, semithick]
(axis cs:102,84737)
--(axis cs:102,505043);

\addplot [semithick, mediumpurple148103189]
table {%
24 33554432
48 562949953421312
96 1.58456325028529e+29
168 7.48288838313422e+50
};
\addplot [semithick, steelblue31119180, mark=*, mark size=1, mark options={solid}, only marks]
table {%
18 0
42 1089.6
90 2067.3
162 304070.000000001
};
\addplot [semithick, darkorange25512714, mark=*, mark size=1, mark options={solid}, only marks]
table {%
22 1138
46 1696.16666666667
94 32937.2
166 188467.6
};
\addplot [semithick, forestgreen4416044, mark=*, mark size=1, mark options={solid}, only marks]
table {%
26 8600.89999999998
50 22681.8
98 189375.6
170 427731.666666665
};
\addplot [semithick, crimson2143940, mark=*, mark size=1, mark options={solid}, only marks]
table {%
30 31892.3
54 75210.0833333333
102 251844.299999999
};
\end{axis}

\end{tikzpicture}


%% file: figures/benchmark_av_res_time.tex
%
\begin{tikzpicture}

\definecolor{crimson2143940}{RGB}{214,39,40}
\definecolor{darkgray176}{RGB}{176,176,176}
\definecolor{darkorange25512714}{RGB}{255,127,14}
\definecolor{forestgreen4416044}{RGB}{44,160,44}
\definecolor{lightgray204}{RGB}{204,204,204}
\definecolor{steelblue31119180}{RGB}{31,119,180}

\begin{axis}[
legend cell align={left},
legend style={
error bar legend,
  fill opacity=0.8,
  draw opacity=1,
  text opacity=1,
  at={(0.98,0.03)},
  anchor=south east,
  draw=lightgray204
},
log basis y={10},
tick align=outside,
tick pos=left,
x grid style={darkgray176},
xlabel={Optimization horizon $T$ (h)},
xmin=-0.31, xmax=3.21,
xtick style={color=black},
xtick={0,1, 2, 3},
xticklabels={24,48, 96, 168},
y grid style={darkgray176},
ymin=0.763305923605485, ymax=23045.9921420311,
ymode=log,
ytick style={color=black},
ylabel={Resolution time (s)}
]
\path [draw=steelblue31119180, semithick]
(axis cs:-0.15,1.2199170589447)
--(axis cs:-0.15,14.8493514060974);

\path [draw=steelblue31119180, semithick]
(axis cs:0.85,1.54920768737793)
--(axis cs:0.85,15.9277083873749);

\path [draw=steelblue31119180, semithick]
(axis cs:1.85,4.05273771286011)
--(axis cs:1.85,31.8070268630981);

\path [draw=steelblue31119180, semithick]
(axis cs:2.85,267.480078220367)
--(axis cs:2.85,3173.38650774956);

\path [draw=darkorange25512714, semithick]
(axis cs:-0.05,1.27233457565308)
--(axis cs:-0.05,9.78963327407837);

\path [draw=darkorange25512714, semithick]
(axis cs:.95,2.50820708274841)
--(axis cs:0.95,20.0158603191376);

\path [draw=darkorange25512714, semithick]
(axis cs:1.95,6.53500461578369)
--(axis cs:1.95,285.139721155167);

\path [draw=darkorange25512714, semithick]
(axis cs:2.95,829.79585647583)
--(axis cs:2.95,3381.53506398201);

\path [draw=forestgreen4416044, semithick]
(axis cs:0.05,1.53253126144409)
--(axis cs:0.05,70.3579578399658);

\path [draw=forestgreen4416044, semithick]
(axis cs:1.05,10.95281291008)
--(axis cs:1.05,316.77522110939);

\path [draw=forestgreen4416044, semithick]
(axis cs:2.05,442.104584932327)
--(axis cs:2.05,3489.55773591995);

\path [draw=forestgreen4416044, semithick]
(axis cs:3.05,7791.82089352608)
--(axis cs:3.05,14419.9494452477);

\path [draw=crimson2143940, semithick]
(axis cs:0.15,2.64020085334778)
--(axis cs:0.15,972.097729444504);

\path [draw=crimson2143940, semithick]
(axis cs:1.15,226.820820569992)
--(axis cs:1.15,5800.44518303871);

\path [draw=crimson2143940, semithick]
(axis cs:2.15,1386.97916269302)
--(axis cs:2.15,12536.6701107025);

\addplot [semithick, steelblue31119180, mark=*, mark size=1, mark options={solid}, only marks]
table {%
-0.15 2.65218722820282
0.85 3.80522921085358
1.85 15.1450152397156
2.85 1311.37046861649
};
\addlegendentry{2 EVs}
\addplot [semithick, darkorange25512714, mark=*, mark size=1, mark options={solid}, only marks]
table {%
-0.05 3.13877234458923
0.95 6.67036004066467
1.95 79.2149761676788
2.95 1432.38655376434
};
\addlegendentry{10 EVs}
\addplot [semithick, forestgreen4416044, mark=*, mark size=1, mark options={solid}, only marks]
table {%
0.05 30.0206317186356
1.05 107.062796139717
2.05 1559.20768713951
3.05 11800.2050302982
};
\addlegendentry{50 EVs}
\addplot [semithick, crimson2143940, mark=*, mark size=1, mark options={solid}, only marks]
table {%
0.15 275.254098105431
1.15 1285.92343790531
2.15 5375.33261811734
};
\addlegendentry{100 EVs}
\end{axis}

\end{tikzpicture}
